\documentclass[review]{elsarticle}
\usepackage{lineno,hyperref}
\usepackage{amssymb, amsmath, amsthm}
\newtheorem{theorem}{Theorem}[section]
\newtheorem{definition}{Definition}[section]
\newtheorem{lemma}{Lemma}[section]
\newtheorem{corollary}{Corollary}[section]

\newtheorem{Remark}{Remark}[section]
\newtheorem{Note}{Note}[section]

\modulolinenumbers[5]

\journal{Journal of \LaTeX\ Templates}
\newcommand{\least}{\let\cs=\@currsize\renewcommand{\baselinestretch}{.9}\tiny\CS}
\begin{document}
	\begin{frontmatter}
		\title{ Invariant bilinear forms under   the rigid motions of a regular polygon  }
		\author{Dilchand Mahto and Jagmohan Tanti*}
		
		
		\address{Department of Mathematics, Central University of Jharkhand, Ranchi-835205, Jharkhand, India\\
			Department of Mathematics, Babasaheb Bhimrao Ambedkar University, Lucknow-226025, Uttar Pradesh, India\\
		*Corresponding author.\\ E-mail: dilchandiitk@gmail.com; jagmohan.t@gmail.com}







\begin{abstract}
 In this paper, for $n$ a  positve integer, we compute the number of $n$ degree representations for a dihedral  group $G$ of order $2m$, $m\geq3$ and the dimensions of the  corresponding spaces of $G$ invariant bilinear forms  over a complex field  $\mathbb{C}$. We explicitly discuss about the existence of a non-degenerate  invariant  bilinear form. The results are important due to their applications in the studies of physical sciences.
\end{abstract}

\begin{keyword}
\texttt  Bilinear  forms; Representation theory;  Vector spaces;  Direct sums; Semi direct product.
\MSC[2010] 15A63\sep  11E04 \sep 06B15 \sep 15A03.
\end{keyword}

\end{frontmatter}


\section{Introduction}
Representation theory enables the study of a group  as operators on certain vector spaces and an orthogonal group with respect to a corresponding bilinear space. Also since last several decades the search of non-degenerate invariant bilinear forms has remained of great interest. Such type of studies acquires an important place in quantum mechanics and other branches of physical sciences. \\
\vskip.5mm
Let $G$ be a finite group and $\mathbb{V}$ a vector space over a field $\mathbb{F}$, then we have the following.
\begin{definition}
	A homomorphism  $\rho$ : $G$ $\rightarrow$ GL($\mathbb{V}$) is called a representation of the group $G$.  $\mathbb{V}$ is also called a representing space of $G$. The dimension of   $\mathbb{V}$ over $\mathbb{F}$ is called degree of the representation $\rho$.
\end{definition}
	Let  $\rho$ : $G$ $\rightarrow$ GL($\mathbb{V}$) be a representation.
\begin{definition}
 A bilinear form  on  $\mathbb{V}$ is said to be invariant under the representation $\rho$  if 
	$$\mathbb{B}(\rho(g)x,\rho(g)y)=\mathbb{B}(x,y), \mbox{ $\forall$ g $\in$ $G$ and x,y $\in\mathbb{V}$}.$$\\
	For the basic properties of a bilinear form one can refer to \cite{Hoff}.
\end{definition}

\noindent Let $\Xi$ denotes the space of bilinear forms on the vector space  $\mathbb{V}$ over $\mathbb{F}$.  
\begin{definition} The space of invariant bilinear forms under the representation $\rho$ is given by
	$$\Xi_{G} = \{\mathbb{B} \in \Xi \,\, |\, \mathbb{B}(\rho(g)x,\rho(g)y)=\mathbb{B}(x,y), \mbox{ $\forall$ $g$ $\in$ $G$ and x,y $\in$ $\mathbb{V}$}\}.$$
	It is obvious that $\Xi_G$ is a subspace of $\Xi$.
\end{definition}
\noindent The  representation ($\rho$, $\mathbb{V}$)  is irreducible \cite{DCT} of degree $n$ if and only if $\{0\}$ and  $\mathbb{V}$ are the only invariant  subspaces of $\mathbb{V}$ under  $\rho$.  Let $r$ be  the  number of  conjugacy classes of  $G$. If  $\mathbb{F}$  is algebraically closed and char( $\mathbb{F}$) is $0$ or not dividing $|G|$, by Frobenius (see pp 319, Theorem (5.9) \cite{Artin})  there are  $r$    irreducible representations $\rho_i$ (say), $1 \leq i \leq r$ of $G$ and $\chi_i$ (say) is the corresponding character of $\rho_i$.  
Also by Maschke's theorem ( see pp 316,  corollary (4.9) \cite{Artin}) every $n$ degree representation of $G$ can be written as a direct sum of copies of irreducible representations. For  $\rho$ = $\oplus_{i=1}^{r}k_{i}\rho_{i}$ an $n$ degree representation of $G$, the coefficient of $\rho_{i}$ is $k_{i}$, $1 \leq i \leq r$, so that $\sum_{i=1}^{r} d_{i}k_{i}=n, $ and $\sum_{i=1}^{r} d_{i}^{2}= |G|, $ where $d_{i}$ is the degree of $\rho_{i}$  and $d_{i}||G|$ \cite{Serr} with  $d_{i'}\geq d_{i}$ when $i'>i$. It is already well understood in the literature that  the   invariant space $\Xi_{G}$ under $\rho$ can be expressed by the set  $\Xi_{G}'$ = $\{X \in \mathbb{M}_{n}(\mathbb{F})\,|\, C_{\rho(g)}^{t}X C_{\rho(g)} = X, \forall g \in G  \}$ with respect to an ordered basis $\underline{e}$ of $\mathbb{V}$, where  $\mathbb{M}_{n}(\mathbb{F})$ is the set of square matrices of order $n$ with entries from $\mathbb{F}$ and  $C_{\rho(g)}=[\rho(g)]_{\underline{e}}$ is the matrix representation of the linear transformation $\rho(g)$ with respect to $\underline{e}$.\\

 In this paper our investigation pertains to the following questions.\\
\\
\textbf{Question.} How many $n$ degree  representations (upto isomosphism) of  $G$  can be  there ?  Distinguish which  of them are  faithful representations   ? What is the dimension of  $\Xi_{G}$  for every $n$ degree representation ? What are the necessary and sufficient conditions for the existence of a non-degenerate invariant bilinear form.\\

These questions have been studied  by many people in the distinct perspectives.  Gongopadhyay and Kulkarni \cite{Gong2} investigated the  existence of T-invariant non-degenerate symmetric (resp. skew-symmetric) bilinear forms. Kulkarni and Tanti \cite{Kulk} investigated the  dimension of space of  T-invariant bilinear forms.  Gongopadhyay,  Mazumder and Sardar \cite{Gong} investigated for  an invertible linear map $T : V \rightarrow V$, when does the vector space
$V$ over $F$ admit a T-invariant non-degenerate c-hermitian form.  Chen \cite{Chen} discussed the all matrix representations of the real numbers. Authors  investigated  the dimensions of invariant spaces and explicitly discussed about the existence of  the non-degenerate invariant bilinear forms under $n$ degree representations of a group of order $p^3$, with prime $p$ [\cite{DCT},\cite{DCJT}]. Sergeichuk \cite{VVSE} studied systems of forms
and linear mappings by associating with them self-adjoint representations of a category with involution. Frobenius \cite{GFRO} proved that every endomorphism of a
finite dimensional vector space V is self-adjoint for at least one non-degenerate symmetric
bilinear form on V. Later, Stenzel \cite{HSTE} determined when an endomorphism could be skew-
self adjoint for a non-degenerate quadratic form, or self-adjoint or skew-self adjoint for a
symplectic form on complex vector spaces. However his results were later generalized to
an arbitrary field \cite{RGTL}. Pazzis \cite{CSPA} tackled the case of the automorphisms of a finite dimensional vector space that are orthogonal (resp. symplectic) for at least one non-degenerate
quadratic form (resp. symplectic form) over an arbitrary field of characteristics 2.\\
Let  $D_m$ be a dihedral group of order $2m, m \in \mathbb{Z^+}$  and   ($\rho$, $\mathbb{V}$)  an $n$ degree representation of $D_m$ over  $\mathbb{C}$.\\

In this paper we investigate about the  counting of $n$ degree representations of   $D_m$ with $m\geq3$ over $\mathbb{C}$,  dimensions of their corresponding  spaces of invariant bilinear forms and establish a characterization criteria for the existence of a non-degenerate invariant  bilinear form. Our  investigations  are stated in the following four main theorems. \\
\begin{theorem}
	The number of  n degree representations (upto isomorphism) of  $D_m$, $m\geq3$   is  $$\sum_{s=0}^{[\frac{n}{2}]} \binom{s+[\frac{m-3}{2}]}{[\frac{m-3}{2}]}\binom{n-2s+2|Z(G)|-1}{2|Z(G)|-1}.$$
	\label{theorem1.1}
\end{theorem}

\begin{theorem}
	The space $\Xi_{D_m}$, with $m \geq 3$  under an n degree representation $(\rho ,\mathbb{V})$ over $\mathbb{C}$ is isomorphic to  the direct sum of  the subspaces $\mathbb{W}_{(G,k_{i}\rho_{i})}$ of $\mathbb{M}_{n}(\mathbb{C})$, i.e., $\Xi_{G}'$ = $\bigoplus_{i=1}^{r} \mathbb{W}_{(G,k_{i}\rho_{i})}$ or  $ \Xi_{G} \cong \bigoplus_{i=1}^{r}\mathbb{M}_{k_i}(\mathbb{C})$, where $\mathbb{W}_{(G,k_{i}\rho_{i})}$ = \{ $ X \in \mathbb{M}_{n}(\mathbb{C})\, |\, X = Diag[O^{11}_{d_{1}k_{1}}, . ... ....,  X_{d_{i}k_{i}}^{ii},  ... .......,  O^{rr}_{d_{r}k_{r}}   ]$ with  $X_{d_{i}k_{i}}^{ii}$  a square matrix of order $d_{i}k_{i} $ satisfying $X_{d_{i}k_{i}}^{ii}=  C_{k_{i}\rho_{i}(g)}^{t}X_{d_{i}k_{i}}^{ii}C_{k_{i}\rho_{i}(g)}$, $\forall g \in D_m$ \}.
	Also   dimension of $\mathbb{W}_{(G,k_{i}\rho_{i})}= k_{i}^{2}$.
	
 

		  
\label{theorem1.2}
\end{theorem}

\begin{theorem}
	
 If $G$ is a dihedral group,  then   an $n$ degree representation $\sigma$ of $G$ 
 admits   a non-degenerate invariant bilinear form if and only if every irreducible representation admits  a non-degenerate invariant bilinear form.\\

\label{theorem1.3}
\end{theorem}

\begin{theorem}
\begin{enumerate}
	\item The number of faithful irreducible representations of a dihedral group $D_m$ is $$[\frac{m-1}{2}]-(l_2+1)(l_3+1)..........(l_p+1)+|Z(G)|+1.$$
	
	\item If $\rho$ is a finite degree representation of $D_m$, then  $\rho(D_m)$  is isomorphic to  either   $\mathbb {Z}_1$, $\mathbb {Z}_2$ or $D_{lcm\{ \frac{m}{(m,t)}|{t^{th }_{\rho}irrep}\}}$, where $t^{th}_{\rho}irrep$  stands for the  $t^{th}$ irreducible representation of degree 2 appearing in $\rho$.
	\end{enumerate}	
\label{thm1.4}  
\end{theorem}
\begin{Remark}
	Thus we get the necessary and sufficient condition for the  existence of a non-degenerate invariant bilinear form under an n degree representation of a dihedral group over $\mathbb{C}$.
\end{Remark}
\section{Preliminaries }
 Symmetries  are the rigid motions of an m-sided  regular  polygon with  $m \in \mathbb{N}$. Considering a  circle $C(0,R)$ with radius $R$ and centered at origin on a real plane, if the point $(R,0)$ is rotated counter- clockwise by an angle $\frac{2\pi}{m}$ radian ( or resp. clockwise) around the  circle, which  returns back to itself after $m$  successive rotations,  while these rotations, the set of points $\{(RCos0,RSin0), (RCos\frac{2\pi}{m},RSin\frac{2\pi}{m}), ..$ $.........., (RCos\frac{2(m-1)\pi}{m},RSin\frac{2(m-1)\pi}{m}) \}$ lies on the circle forming the  vertices of an  m-sided regular  polygon $P_m$ say, which is a symmetric regular polygon. Let $a$ be the rotation by angle $\frac{2 \pi}{m}$ counter-clockwise then for $1 \leq s \leq m$, rotating it counter-clockwise by an   angle of $\frac{2s\pi}{m}$ radian, we get  $"a^s "$. Thus we get $m$  symmetric regular  m-gons namely $a,a^2, a^3, .... ...., a^{m-1}, a^m =P_m$, the last one $a^m$ is the  original regular m-gon. When $m$ is  odd,  for   $1 \leq s \leq m$,  the reflection about the angle bisector of the vertex $(RCos\frac{2s\pi}{m},RSin\frac{2s\pi}{m})$ of  the regular m-gon $a^s$ is also a symmetric regular m-gon $b$ (say), whereas if $m$ is even, the reflections about  the side bisectors as well as  the angle bisectors of the  vertices $(RCos\frac{2s\pi}{m},RSin\frac{2s\pi}{m}), 1 \leq s \leq m$ of  the regular m-gon $a^s$ is symmetric with $a^m$ \cite{KCON}. Thus there is a set of   $2m$ symmetric regular m-gons $\{a,a^2 ,........,a^m, b, ab,a^2b, .........,a^{m-1}b\}$, which  forms a group with the operation being compositions of m-gons and known as a dihedral group of order $2m$   denoted by $D_m$.  Centre of $D_m$,    $Z(D_m)= \{a^m\}$ or $ \{a^m, a^{\frac{m}{2}}\}$ according to $m$ is odd or even respectively and the number of the  conjugacy classes is $r= 2|Z(D_m)| +[\frac{m-1}{2}]$, which is equal to the  number of irreducible representations (over an algebraicaly closed field $\mathbb{F}$, char($\mathbb{F})=0$ or not dividing $2m$)   with degree $d_i$. Here  $d_i | |G|$ and $\sum_{i=1}^{r} d_i^2 =|G|$.  There are $ 2|Z(D_m)|$ representations of degree $1$ and $\big[\frac{m-1}{2}\big]$  representations of degree $2$ for  $D_m$. In the next section, we formulate all $r$  irreducible representations of $D_m$  such that $C_{\rho_i(g)}$ is either a rotation operator or a reflection operator. \\

\begin{definition}
	The  character  of $\rho$ is a  function    $\chi$ : $G$ $\rightarrow$ $\mathbb{F}$,  $\chi(g)=$ tr( $[\rho(g)]_{\underline{e}}$ ) and is also called character of the group $G$.
\end{definition}
\begin{theorem}
	(Maschke's Theorem): If char($\mathbb{F}$) does not divide $|G |$, then every representation of  $G$ is a direct sum of irreducible representations.
\end{theorem} 

\begin{proof} 
	See pp 316,   corollary (4.9)  \cite{Artin}.
\end{proof}

\begin{theorem}
	Two representations ($\rho, \mathbb{V}$) and ($\rho', \mathbb{V}$) of $G$ are  isomorphic iff their character tables are same i.e, $\chi(g)=\chi'(g)$ for all g $\in G$.
	\label{theorem2.2}
\end{theorem}

\begin{proof}
	See  pp 319,  corollary ( 5.13) \cite{Artin}.
\end{proof}
 In the rest part of this section we take $\mathbb{F}=\mathbb{C}.$
\subsection{Irreducible representations of  $D_m$ with  $m \geq3$. }
\label{section3}
\noindent In this subsection  $(\rho_i, \mathbb{W}_{d_i}) $  stands for  an  irreducible representation of $D_m$  with degree $1$ or $2$  over $\mathbb{C}$.
Let $\rho_{2|Z(G)|+t}$, denotes an   irreducible representations of  degree $2$ of $G$, where $1 \leq t \leq [\frac{m-1}{2}]$.  Since  $\rho_{2|Z(G)|+t}$ is an homomorphism from $D_m$ to $GL(\mathbb{W}_2) \cong GL(2, \mathbb{F})$. So by the fundamental  theorem of homomorphism $\frac{G}{Ker(\rho_{2|Z(G)|+t})} \cong $ $\rho_{2|Z(G)|+t}(G)$. 	Now we calculate those $t$ which are co-prime with $m$. All   $2m$, regular  m-gons are in\\
\begin{center}
	$D_{m}$ = $\{ \bold{1}, a, a^2, .........., a^{m-1}, b, ab, a^2b ,... ... , a^{m-1}b\,|\, a^m= b^2 =\bold{1}, ba= a^{m-1}b \}$.
	\label{table1}
\end{center}
The representation  $\rho_{2|Z(G)|+t}$ rotates a  regular m-gon $a^s$  in the counter-clockwise  direction by an angle  $s\theta_t$, where  $\theta_t = t\frac{2\pi}{m}$ (or resp. clockwise $(m-t)\frac{2\pi}{m}$) and $b$  the reflection  about $x-axis$.
For $1 \leq s\leq m$, the polygon $a^sb$ is mapped to the  composition of   rotation  by angle $s\theta_t$ and the reflection     under an irreducible representation of degree $2$, whose trace is zero. Thus the character of the group is to be calculated on the  regular m-gon  $a^s$, $1\leq s\leq m$.  \\
 By the fundamental theorem of arithmetic, we have
 $$m= 2^{l_2} 3^{l_3}5^{l_5}7^{l_7}11^{l_{11}}.......  p^{l_p},$$
 where $p$ is the largest   prime divisor of  $m$, so $l_p \geq 1$. The number of divisors of $m$ excluding $1$ is $(l_2+1)(l_3+1)(l_5+1)(l_7+1)(l_{11}+1).....(l_p+1)-1$. We need   divisors of $m$ greater than $1$ and  less than $\frac{m}{2}$. These divisors  fall inside the range of  $t$ which are  not coprime to $m$, also for each $t$ we have an irreducible representation of degree $2$. If $t_1$
 is a divisor of $m$ then $\frac{m}{(m,t_1)}$ is  the order of $\sigma_{2|Z(G)|+t_1}(a)$, so  kernel of the respective irreducible representation is non-trivial.  There may be many such distinct divisors of $m$,  less than $\frac{m}{2}$.  Thus the  number of non-isomorphic images of $G$ under the representations $\sigma_{2|Z(G)|+t}$ is same as  the number of distinct divisors of  $m$, less than $\frac{m}{2}$ 
and whenever  $(m,t)=1$ then we have   $|Ker(\rho_{2|Z(G)|+t})|= 1 $. \\

\subsection{Counter-clockwise rotations and their compositions with reflection can be seen as below.}
$$ \rho_{2|Z(G)|+t}(a^s)= \begin{bmatrix}
Cos( \frac{2\pi}{m}ts)&-Sin( \frac{2\pi}{m}ts) \\
Sin( \frac{2\pi}{m}ts) &Cos( \frac{2\pi}{m}ts) 
\end{bmatrix}  and \hspace{0.1cm} \rho_{2|Z(G)|+t}(a^sb)=\begin{bmatrix}
Cos( \frac{2\pi}{m}ts)&Sin( \frac{2\pi}{m}ts) \\
Sin( \frac{2\pi}{m}ts) &-Cos( \frac{2\pi}{m}ts) 
\end{bmatrix}.$$
\subsection {For $m$ is  odd  and  $1 \leq t \leq \frac{m-1}{2}$, all irreducible representations of  $G$ are recorded in the following  table.}

\begin{center}
	 
	\begin{tabular}{cccccccc}
		{\bf Table 1}.&Irreducible& representations of $D_m$  when $m$ is  odd and $1\leq t\leq\frac{m-1}{2}$.\\
		\hline
		& $\rho_1$&$\rho_2$& $\rho_{2+t}$\\
		\hline
$a$& 1&1& $\begin{bmatrix}
Cos(\theta_t)&-Sin(\theta_t) \\
Sin(\theta_t) &Cos(\theta_t) 
\end{bmatrix}$\\
		\hline
	b& 1&$-1$& $\begin{bmatrix}
	1&0 \\
	0 &-1
	\end{bmatrix}$\\
	\hline
	\end{tabular}
\label{3.1}	
\end{center}

\subsection{For $m$ is  even  and   $1 \leq t \leq \frac{m}{2}-1$,  all irreducible representations of $G$ are  presented by the following  table.} 

\begin{center}
	
	\begin{tabular}{cccccccc}
		{\bf Table 2}.&Irreducible& representations of $D_m$  when $m$ is  even and $1\leq t\leq\frac{m}{2}-1$.\\
		\hline
		& $\rho_1$&$\rho_2$&$\rho_3$&$\rho_4$& $\rho_{4+t}$\\
		\hline
		$a$& 1&$1$&$-1$&$-1$& $\begin{bmatrix}
		Cos(\theta_t)&-Sin(\theta_t) \\
		Sin(\theta_t) &Cos(\theta_t) 
		\end{bmatrix}$\\
		\hline
		b& 1&$-1$&$-1$&$1$& $\begin{bmatrix}
		1&0 \\
		0 &-1
		\end{bmatrix}$\\
		\hline

	\end{tabular}
\label{3.2}	
\end{center}

\noindent Now as
\begin{equation}
\rho= k_{1} \rho_{1}\oplus k_{2} \rho_{2}\oplus  ............. \oplus k_{r} \rho_{r},
\label{directsum}
\end{equation}
where for every $1 \leq i \leq r$,  $k_{i} \rho_{i}$ stands for the direct sum of  $k_{i}$ copies of the irreducible representation $\rho_{i}$.

\noindent Let $\chi$ be the corresponding character of the representation $\rho$, then  
$$ \chi= k_{1} \chi_{1}+ k_{2} \chi_{2}+  ............. + k_{r} \chi_{r},$$
where $\chi_{i}$ is the irreducible character of $\rho_{i}$,  for every $i$, $1 \leq i \leq r$.
Dimension of the character $\chi$ is being calculated at the identity element of a group. i.e,
\[ dim(\rho) = \chi(1)=tr(\rho(1)).\]
\begin{equation}
\implies d_{1} k_{1} + d_{2}k_{2} +............. +d_{r}k_{r}= n. 
\label{splitn}
\end{equation}
\begin{Note}
	equation (\ref{splitn}) holds in more general case which  helps us in finding all  possible distinct r-tuples ($k_{1}, k_{2}, ......, k_{r}$), which correspond to the distinct n degree  representations (up to isomorphism) of a finite group.
\end{Note}
The orthonormality condition of characters  of   irreducible representations of degree $2$ is being  calculated in the following manner:
$$(\chi_{t_1+2|Z(D_m)|},\chi_{t_2+2|Z(D_m)|})=
\begin{cases}

\frac{4}{|D_m|}\sum_{a^s \in <a>} Cos^2(s\theta_{t_1}), \mbox{  when $t_1=t_2$ }, \\
\\
\frac{4}{|D_m|}\sum_{a^s \in <a>} Cos(s\theta_{t_1})Cos(s\theta_{t_2}),  \mbox{  when $t_1\neq t_2$ . }\\
\end{cases} 
$$
$$=
\begin{cases}
	
1, \mbox{  when $t_1=t_2$ }, \\
	\\
0,  \mbox{  when $t_1\neq t_2$ . }\\
\end{cases} 
$$
\section{Faithful representations of a dihedral group  of order $2m$ with  $m \geq3$. }
In this section we distinguish all   $n$ degree faithful representations of $G=D_m$ over  $\mathbb{C}$.  Here $m= 2^{l_2} 3^{l_3}5^{l_5}7^{l_7}.......  p^{l_p}.$ 
\begin{theorem}
The number of irreducible   representations of a dihedral group $G$  of order $2m$ with non-trivial kernels is $(l_2 +1) (l_3+1).....(l_p+1)+|Z(G)|-1$.
\label{th3.1}
\end{theorem}
\begin{proof}
 Number of irreducible representations of  $D_m$ is $r=2|Z(G)|+ [\frac{m-1}{2}]$. For a 2 degree irreducible representation $\rho_{2|Z(G)|+t}$,   if  $(m,t)\neq1$ then kernel is non-trivial. The    number of  divisors of $m$  excluding  $1$,  less than $\frac{m}{2}$ is $(l_2 +1) (l_3+1).....(l_p+1)-|Z(G)|-1$ and  $2|Z(G)| $  representations of degree $1$ have non-trivial kernel. Thus the result follows. 
\end{proof}
\begin{theorem}
	If $\rho_{2|Z(G)|+t}$ is an irreducible representation of degree $2$ of a dihedral group $G$ of order $2m$ and $(m,t)=1$, then $\rho_{2|Z(G)|+t}$ is a faithful representation.\\
	 	\label{th3.2}
\end{theorem}
\begin{proof}
For $1 \leq t \leq [\frac{m-1}{2}]$,  $\rho_{2|Z(G)|+t}$ is a homomorphism from   $G$  to $GL(\mathbb{W}_2)$. Let $t_1 \in \{1,2, ... ... ..., [\frac{m-1}{2}]\}$ such that $(m,t_1)=1$ then $Ker(\rho_{2|Z(G)|+t_1})=\{\bold{1} \}$. Thus the result follows. 
\end{proof}
\begin{corollary}
The number of faithful irreducible  representations of a dihedral group $G$ of order $2m$ is  $[\frac{m-1}{2}]-(l_2+1)(l_3+1)..........(l_p+1)+|Z(G)|+1$.
\label{corfaithful}
	\end{corollary}
\begin{proof}
Follows from the theorems \ref{th3.1} and \ref{th3.2}.
\end{proof}
\begin{theorem}
If $\sigma$ is a non-trivial  one degree representation of $D_m$, then the image of $D_m$ under $\sigma$ is isomorphic to   $\mathbb {Z}_2$.
\label{th3.3}
\end{theorem}
\begin{proof}
Follows from the tables in the  subsections \ref{3.1} and \ref{3.2}.
\end{proof}
\begin{theorem}
If $\sigma$ is an irreducible representation for  $D_m$ of degree $2$, then the image of $D_m$ under $\sigma$ is isomorphic to  $D_{\frac{m}{(m,t)}}$.
$\label{th3.4}$
\end{theorem}
\begin{proof}
If  greatest common divisor of $m$ and $t$ is $(m,t)$ then $\frac{m}{(m,t)}$ is the  order of $\sigma(a)$ and $\sigma(b)$ is the reflection operator. Thus the group $\sigma(D_m)$ is isomorphic to $D_{\frac{m}{(m,t)}}$.
\end{proof}

\begin{corollary}
If $\sigma$ is a non-trivial irreducible representation of $D_m$, then the image of $D_m$ under $\sigma$ is isomorphic to either  $\mathbb {Z}_2$ or $D_{\frac{m}{(m,t)}}$. 
\end{corollary}
\begin{proof}
 Follows from the   theorems \ref{th3.3} and \ref{th3.4}.
\end{proof}
\begin{corollary}
	 The number of non-isomorphic images under irreducible representations of $D_m$ is $2+(l_2 +1) (l_3+1).....(l_p+1)-|Z(D_m)|$.
\end{corollary}
\begin{proof}
An irreducible  representation sends  $D_m$ to either one of  $\mathbb{Z}_1$,  $\mathbb{Z}_2$,  $D_{\frac{m}{(m,t)}}$  and as the number of distinct divisors of $m$, less than $\frac{m}{2}$ is $(l_2 +1) (l_3+1).....(l_p+1)-|Z(D_m)|$, the result follows.
\end{proof}

\begin{corollary}
	If $\rho$ is a finite degree representation of $D_m$, then  $\rho(D_m)$  is isomorphic to either  $\mathbb {Z}_1$, $\mathbb {Z}_2$ or $D_{lcm\{ \frac{m}{(m,t)}|{t^{th }_{\rho}irrep}\}}$, where $t^{th}_{\rho}irrep$  stands for the  $t^{th}$ irreducible representation of degree 2 appearing in $\rho$.
	\label{corirred} 
\end{corollary}
\begin{proof}
	If $\rho$ consists of  only degree $1$ representations  then $\rho(D_m) \cong \mathbb {Z}_1$ or  $\mathbb {Z}_2$. So suppose $\rho$ consists of  $t_{r_1}^{th},t_{r_2}^{th},.......,t_{r_l}^{th}$ irreducible representations of degree 2 then from Theorem  $\ref{th3.4}$, for every $t_r$, $\rho_{t_r}(D_m)$ is isomorphic to $D_{\frac{m}{(m,t_r)}}$ i.e.,  ${\frac{m}{(m,t_r)}}$ is the order of  $\rho_{t_r}(a)$ . 
	Therfore $lcm{\{\frac{m}{(m,t_{r_1})}, \frac{m}{(m,t_{r_2})}, ......, \frac{m}{(m,t_{r_l})}\}}$ is the order of $\rho(a)$ together with $\rho(b)^2$ is an identity operator. Hence the result follows. 
	\end{proof}
\section{Existence of non-degenerate invariant  bilinear  forms.}
\indent An element in  the space of invariant  bilinear forms under representation of  a finite  group is either    non-degenerate or   degenerate. An  element of the space is degenerate when at least one irreducible representation  consists of degenerate bilinear forms. If all elements of the space is degenerate then  the space is called a degenerate invariant space, which is  also  discussed in \cite{DCT}  for the groups of order $8$ and \cite{DCJT} for the groups of order $p^3$, with an odd prime $p$. How many such representations exists out of  total representations, is a matter of investigation. Some of the spaces contains both non-degenerate and degenerate invariant  bilinear forms under   a certain representation. In this section we compute the number of such representations of $D_m$ over $\mathbb{C}$.\\
\begin{Remark}
	The space $\Xi_{G}'$  of invariant bilinear forms under an n degree representation $\rho$ contains only those $X \in \mathbb{M}_{n}(\mathbb{C})$ whose  $(i,j)^{th}$ block is a O sub-matrix of order ${d_{i}k_{i}\times d_{j}k_{j}}$ when $i \neq j$  whereas  the $(i,i)^{th}$ block of X, for $1 \leq i \leq 2|Z(D_m)|$, is given by
	$$X^{ii}_{d_ik_i}= 
	\begin{bmatrix}
	x^{ii}_{11}&x^{ii}_{12}&...&...&...&x^{ii}_{1k_{i}}\\
	x^{ii}_{21}&x^{ii}_{22}&...&...&...&x^{ii}_{2k_{i}}\\
	...&...&...&...&...&...\\
	...&...&...&...&...&...\\
		...&...&...&...&...&...\\
	x^{ii}_{k_{i}1}&x^{ii}_{k_{i}2}&...&...&...&x^{ii}_{k_{i}k_{i}}\\
	\label{remark4.1}
	\end{bmatrix}.$$

	\label{remark}
and 	  for  $i \geq 2|Z(D_m)|+1$, 

$$X^{ii}_{d_ik_i}= 
\begin{bmatrix}
x^{ii}_{11}I_{2}&x^{ii}_{13}I_{2}&...&...&...&x^{ii}_{1(k_{i}-1)}I_{2}\\
x^{ii}_{31}I_{2}&x^{ii}_{33}I_{2}&...&...&...&x^{ii}_{3(k_{i}-1)}I_{2}\\
...&...&...&...&...&...\\
...&...&...&...&...&...\\
...&...&...&...&...&...\\
x^{ii}_{(2k_{i}-1)1}I_{2}&x^{ii}_{(2k_{i}-1)2}I_{2}&...&...&...&x^{ii}_{(2k_{i}-1)(2k_{i}-1)}I_{2}\\
\end{bmatrix}.
$$
\end{Remark} 
\begin{Note}
	 $X\in \Xi_G'$ is an invariant bilinear form under $\rho$ if and only if for every $i,  1 \leq i \leq r$,  $X_{d_ik_i}^{ii}$ is an invariant bilinear form under $k_i\rho_i$. 
\end{Note}
\subsection{Characterization of invariant bilinear forms under an n degree representation of a dihedral group of order $2m$, with $m \geq3$.}

\begin{lemma}
If $X \in \Xi_{G}'$, and  for $1 \leq i \leq r$, $X^{ii}_{d_{i}k_{i}}$ is a non-singular sub-matrix, then $X$ must be non-singular.
	\label{singular}
\end{lemma}  

\begin{proof}  With reference to the above remark, for every $X \in \Xi_{G}'$, we have $X= Diag\big[X^{11}_{d_{1}k_{1}}, X^{22}_{d_{2}k_{2}},......, X^{ii}_{d_{i}k_{i}},....., X^{rr}_{d_{r}k_{r}} \big]$ with $X^{ii}_{d_{i}k_{i}}=  C_{k_{i}\rho_{i}(g)}^{t}X^{ii}_{d_{i}k_{i}}C_{k_{i}\rho_{i}(g)}$ and  for $1 \leq i \leq r$, $X^{ii}_{d_{i}k_{i}}$ is a non-singular sub-matrix, so  there exists  $Y= Diag\big[(X^{11}_{d_{1}k_{1}})^{-1}, (X^{22}_{d_{2}k_{2}})^{-1},......, (X^{ii}_{d_{i}k_{i}})^{-1},....., (X^{rr}_{d_{r}k_{r}})^{-1} \big]$ in $ \mathbb{M}_{n}(\mathbb{C})$ such that $XY= I_n=YX$. Thus the result follows. 
 \end{proof}

		 To prove  the next lemma we will choose  only those $X \in \mathbb{M}_{n}(\mathbb{C})$ whose $(i,j)^{th} $ block matrix is zero for $i \neq j$ and   for  the $(i,i)^{th}$ block-diagonal matrix   $X^{ii}_{d_{i}k_{i}}$ is    non-singular. 

 \begin{lemma} For $n \in \mathbb{Z}^{+}$, every  $n$-degree representation of a dihedral  group $G$, has a non-degenerate invariant  bilinear form.
 	\label{nondeg}
 \end{lemma}
\begin{proof} From equation (\ref{splitn}) we have $d_{1} k_{1} + d_{2}k_{2} +............. +d_{r}k_{r}= n$ and  $X \in \mathbb{M}_{n}(\mathbb{C})$ such that $X= Diag\big[X^{11}_{d_{1}k_{1}}, X^{22}_{d_{2}k_{2}},$ $......, X^{ii}_{d_{i}k_{i}},....., X^{rr}_{d_{r}k_{r}} \big]$. If for every  $i , 1 \leq i\leq r$, the block diagonal sub-matrix  $X^{ii}_{d_{i}k_{i}}$  of $X$ is  chosen  (from the  above remark   $\ref{remark}$ ) to be  non-singular,  then  $X^{ii}_{d_{i}k_{i}}=  C_{k_{i}\rho_{i}(g)}^{t}X^{ii}_{d_{i}k_{i}}C_{k_{i}\rho_{i}(g)}$,  $\forall g \in G$.  Therefore  $X \in \Xi_{G}' $  and is non-singular.
\end{proof}
 \begin{lemma}  Every irreducible representation  of    a dihedral group consists of a non-degenerate invariant bilinear form. 
	\label{lemma43}
\end{lemma}
\begin{proof} Follows from the proof of the lemma $\ref{nondeg}$ .
\end{proof}


\begin{Remark} Since $\mathbb{C}$ contains  infinitely many non zero elements, hence  if there is one non-degenerate  invariant bilinear form in the space $\Xi_{G}$, it has   infinitely many.
\end{Remark}
 Thus from Lemma  $\ref{nondeg}$, we find that every $n$ degree representation of a dihedral  group $G$ of order $2m$, with $m \geq 3$,    consists of a non-degenerate invariant  bilinear form.
 

\begin{lemma}
	Let $G$ be a dihedral group of order $2m$ and $\rho= \oplus_{i=1}^{r} k_{i}\rho_{i}$  an n degree representation of $G$, then $\rho$ has a degenerate invariant  bilinear form iff at least one block-diagonal matrix   is singular. 
	\label{lemma3.9}
\end{lemma}

\begin{proof} 
	Easy to see. 	
\end{proof}

\begin{definition}
	The space $\Xi_{G}$ of invariant bilinear forms  is   called  degenerate  if it's all elements are degenerate.
\end{definition}
We  will  discuss about the degenerate  invariant  space in the later section.\\

\section{Dimensions of  spaces of  invariant bilinear forms under representations of a  dihedral group of order $2m$.}
\noindent  The  space of invariant   bilinear forms under an $n$ degree representation is generated by finitely many vectors so its  dimension is finite along with its  symmetric subspace and the skew-symmetric subspace. In this section  we calculate  the dimension of the space of invariant bilinear forms under a representation over $\mathbb{C}$ of a dihedral  group $D_m$ of order $2m$, with $m \geq 3$.\\
\begin{theorem}
	If $\Xi_{G}$ is the  space of invariant bilinear forms  under an n degree representation  $\rho =\oplus_{i=1}^{r}k_{i}\rho_{i}$ of $D_m$, then  dim$(\Xi_{G})= \sum_{i=1}^{2|Z(D_m)|+[\frac{m-1}{2}]}k_{i}^{2}$. 	 
	\label{th5.1}
\end{theorem}
\begin{proof}
For every $X \in \Xi_{G}'$, we have $X= Diag\big[X^{11}_{d_{1}k_{1}}, X^{22}_{d_{2}k_{2}},......, X^{ii}_{d_{i}k_{i}},....., X^{rr}_{d_{r}k_{r}} \big]$ with $X^{ii}_{d_{i}k_{i}}=  C_{k_{i}\rho_{i}(g)}^{t}X^{ii}_{d_{i}k_{i}}C_{k_{i}\rho_{i}(g)}$, for $1 \leq i \leq r$  and to generate these sub-matrices it needs $k_i^2$ vectors from $\Xi_G'$. Thus the  result follows.	
\end{proof}
\begin{corollary}
	The  space of invariant symmetric bilinear forms  under an n degree representation  $\rho =\oplus_{i=1}^{r}k_{i}\rho_{i}$ of  $D_m$ has dimension  $= \sum_{i=1}^{2|Z(D_m)|+[\frac{m-1}{2}]}\frac{k_{i}(k_{i}+1)}{2}$.
\end{corollary}
\begin{proof}
	Follows  from the proof of theorem $\ref{th5.1}$.
\end{proof}

\begin{corollary}
	The space of invariant skew-symmetric bilinear forms  under an n degree representation  $\rho =\oplus_{i=1}^{r}k_{i}\rho_{i}$ of  $D_m$ has  dimension $=$ $\sum_{i=1}^{2|Z(D_m)|+[\frac{m-1}{2}]}\frac{k_{i}(k_i-1)}{2}$.
\end{corollary}

\begin{proof}
	Follows  from the proof of theorem $\ref{th5.1}$.
\end{proof}


\section{Main results \label{proof1.1}}
Here we present the proofs of the  main theorems stated in the Introduction section.
\begin{flushleft}
\textbf{Proof of theorem $\ref{theorem1.1}$}
 Given $G$ is the  dihedral  group of order $2m, m\geq3$  (so 
   $r =[\frac{m-1}{2}]+2|Z(G)|$) and degree of the representation $\rho$	is $n$. Also $d_{i}$ =1 for $1\leq i \leq 2|Z(G)|$ $\&$ $d_{i}=2$ for $2|Z(G)|+1 \leq i \leq [\frac{m-1}{2}]+2|Z(G)|$. Now from equation (\ref{splitn}), we have 
$$ \hspace{1.2cm} k_{1} +.....+ k_{2|Z(G)|} +2k_{2|Z(G)|+1}+.................. +2k_{[\frac{m-1}{2}]+2|Z(G)|} =n. $$
$$ i. e., \hspace{1.2cm} k_{1} +.....+ k_{2|Z(G)|} =n-2(k_{2|Z(G)|+1}+.................. +k_{[\frac{m-1}{2}]+2|Z(G)|}). $$
As $(k_{2|Z(G)|+1}+.................. +k_{[\frac{m-1}{2}]+2|Z(G)|}) \leq [\frac{n}{2}]$,  we have $[\frac{n}{2}]+1$ equations stated below  \\

\begin{equation}
k_{2|Z(G)|+1}+.................. +k_{[\frac{m-1}{2}]+2|Z(G)|}=0.
\end{equation}
\begin{equation}
k_{2|Z(G)|+1}+.................. +k_{[\frac{m-1}{2}]+2|Z(G)|}=1.
\end{equation}
\begin{equation}
k_{2|Z(G)|+1}+.................. +k_{[\frac{m-1}{2}]+2|Z(G)|}=2.
\end{equation}
$$.........................................................................$$
$$.........................................................................$$
$$.........................................................................$$
\begin{equation}
k_{2|Z(G)|+1}+.................. +k_{[\frac{m-1}{2}]+2|Z(G)|}= \Big[\frac{n}{2}\Big].
\end{equation}

The number of distinct solutions to  each of these equations  is $\binom{s+[\frac{m-3}{2}]}{[\frac{m-3}{2}]}$, $0 \leq s \leq [\frac{n}{2}]$.\\
Thus the number of all  distinct    $2|Z(G)|+[\frac{m-1}{2}]$ tuples ($k_{1},....., k_{2|Z(G)|}, ......., k_{2|Z(G)|+[\frac{m-1}{2}]}$) is $\sum_{s=0}^{[\frac{n}{2}]} \binom{s+[\frac{m-3}{2}]}{[\frac{m-3}{2}]}\binom{n-2s+2|Z(G)|-1}{2|Z(G)|-1}$.\\

\end{flushleft}

Thus from equation (\ref{splitn}) and Theorem $\ref{theorem2.2}$ the  number of  $n$ degree representations (upto isomorphism) of the dihedral  group $D_m, m\geq3$  is $\sum_{s=0}^{[\frac{n}{2}]} \binom{s+[\frac{m-3}{2}]}{[\frac{m-3}{2}]}\binom{n-2s+2|Z(G)|-1}{2|Z(G)|-1}$.$\hspace{2.8in}$ $\square$\\

\noindent \textbf{Proof of  theorem $\ref{theorem1.2}$}
 Let $X$ be an element of $\Xi_{G}'$ then we have

\[C_{\rho(g)} ^{t}X C_{\rho(g)}= X \,\,and\,\, X=Diog[X^{11}_{d_{1}k_{1}},X^{22}_{d_{2}k_{2}},  ....... , X^{ii}_{d_{i}k_{i}}, .....,  X^{rr}_{d_{r}k_{r}} ]. \]
 
\noindent Existence:\\
 Let $X \in \Xi_{G}' $ then for  $1\leq i \leq r$,  there exists at least one  $X_{i}=Diog[O^{11}_{d_{1}k_{1}},O^{22}_{d_{2}k_{2}},$  ....... $, X^{ii}_{d_{i}k_{i}}, .....,  O^{rr}_{d_{r}k_{r}} ] \in \mathbb{W}_{(G,k_{i}\rho_{i})}$, such that  $\sum_{i=1}^{r} X_{i} =X$.\\

\noindent Uniqueness: \\
For  $1\leq i \leq r$, suppose there exists  $Y_{i} \in \mathbb{W}_{(G,k_{i}\rho_{i})} $, such that   $\sum_{i=1}^{r} Y_{i} =X$, then  $\sum_{i=1}^{r} X_{i} $ = $\sum_{i=1}^{r} Y_{i}$ i.e,  $Y_{j}-X_{j}$= $\sum_{i=1, i\neq j}^{r} (X_{i}-Y_{i})$. Therefore $Y_{j}-X_{j} \in $ $\sum_{i=1, i\neq j}^{r} \mathbb{W}_{(G,k_{i}\rho_{i})}  $ hence  $Y_{j}-X_{j}$ = O or $Y_{j} = X_{j}$ for all $j$.\\Thus we have 
\begin{equation}
 \Xi_{G}' =\oplus_{i=1}^{r} \mathbb{W}_{(G,k_{i}\rho_{i})} \hspace{0.1cm} and \hspace{0.1cm} dim(\Xi_{G}') = \sum_{i=1}^{r} dim(\mathbb{W}_{(G,k_{i}\rho_{i})}).
\label{wg}
\end{equation}

\begin{flushleft}
Now 
as $\mathbb{W}_{(G,k_{i}\rho_{i})} = \{ X \in \mathbb{M}_{n}(\mathbb{C})\, |\, X = Diag[O^{11}_{d_{1}k_{1}}, . ....... ....,  X_{d_{i}k_{i}}^{ii},  ... ..........,  O^{rr}_{d_{r}k_{r}}   ]$ with  $X_{d_{i}k_{i}}^{ii}$  a square sub - matrix of order $d_{i}k_{i} $ satisfies $X_{d_{i}k_{i}}^{ii}=  C_{k_{i}\rho_{i}(g)}^{t}X_{d_{i}k_{i}}^{ii}C_{k_{i}\rho_{i}(g)}$, $\forall g \in G$ \},
 from the remark  \ref{remark4.1}, we see that for  $1\leq i \leq 2|Z(G)|+[\frac{m-1}{2}]$, the sub-matrices  $X_{d_{i}k_{i}}^{ii}$ in  $\mathbb{W}_{(G,k_{i}\rho_{i})} $  have $k_i^{2}$ free variables $\&$  $\mathbb{W}_{(G,k_{i}\rho_{i})} \cong $ $\mathbb{M}_{k_i}(\mathbb{C})$. Thus $\Xi_{G}'\cong $ $\oplus_{i=1}^{r}\mathbb{M}_{k_1}(\mathbb{C})$ and $dim(\mathbb{W}_{(G,k_{i}\rho_{i})} )= k_{i}^{2}$.
\end{flushleft}


Thus  substituting this in equation  (\ref{wg}) we get the dimension of $\Xi_{G}'$. \\

\noindent \textbf{Proof of  theorem $\ref{theorem1.3}$} Follows immediately from  Lemmas $\ref{singular}$ to $\ref{lemma43}$ .$\hspace{2.5 in}$ $\square$\\

\noindent \textbf{Proof of  theorem $\ref{thm1.4}$}
Follows immediately from the corollaries $\ref{corfaithful}$ and $\ref{corirred}$.
\subsection{Degenerate invariant  spaces} \noindent From  Theorem $\ref{theorem1.1}$ and Lemma $\ref{nondeg}$, for every $n \in \mathbb{Z^+}$, an $n$ degree representation of $D_m$ has a non- degenerate invariant bilinear forms. If the block-diagonal matrix $X_{d_i k_i}^{ii}$ of $X \in \Xi'_{G}$ is singular at least for one $i$ then the invariant bilinear form $X$ is singular. \\

\vskip2mm
Thus here we have completely characterised the representations of a dihedral group of order $2m$, $m \geq 3$ to admit a non-degenerate invariant bilinear form over complex  field. The authors hope to return in future to the same work over an arbitrary field and for some different finite groups.
\vskip2mm
\noindent {\bf Funding}: Not applicable.
\vskip2mm
\noindent {\bf Conflicts of interest/Competing interests}: Not applicable.
\vskip2mm
\noindent {\bf Availability of data and material}: The manuscript has no associated data.
\vskip2mm
\noindent {\bf Code availability}: Not applicable.
\vskip2mm
\noindent {\bf Acknowledgement} The first author would like to thank UGC, India for providing the research fellowship and the Central University of Jharkhand, India for support to carry out this research work. The second author is thankful to the Babasaheb Bhimrao Ambedkar university for providing excellent environment to finalize this work.

\end{document}